\documentclass{amsart}
\usepackage{amssymb}
\usepackage{amsmath}
\usepackage{setspace}
\usepackage[super]{nth}
\usepackage{xcolor}
\newtheorem{theorem}{Theorem}[section]

\newtheorem{lemma}[theorem]{Lemma}

\theoremstyle{definition}

\newtheorem{conjecture}[theorem]{Conjecture}
\newtheorem{remark}[theorem]{Remark}

\numberwithin{equation}{section}

\begin{document}
\title[An explicit evaluation of $\nth{10}$-power moment of quadratic Gauss sums]
{An explicit evaluation of $\nth{10}$-power moment of \\ quadratic Gauss sums and some applications}


\author{Nilanjan Bag}
 \address{Department of Mathematics, Harish-Chandra Research Institute, HBNI, Chhatnag Road, Jhunsi, Prayagraj (Allahabad) - 211 019, India}
\email{nilanjanbag@hri.res.in}
\author{Antonio Rojas-León}
 \address{Departament of Algebra, Universidad de Sevilla, c/Tarfia, s/n, 41012 Sevilla, Spain}
\email{arojas@us.es}

\author{Zhang Wenpeng}
 \address{School of Mathematics, Northwest University, Xi'an, 710127, Shaanxi, P. R. China}
\email{wpzhang@nwu.edu.cn}
\date{16th April, 2021}
\begin{abstract}
In this paper we have estimated one multi-variable character sum
\begin{align*}
\sum_{a=2}^{p-2}\sum_{b=1}^{p-1}\sum_{c=2}^{p-2}\sum_{d=1}^{p-1}\left(\frac{a^2-b^2}{p}\right)\left(\frac{b^2-1}{p}\right)
  \left(\frac{c^2-d^2}{p}\right)\left(\frac{d^2-1}{p}\right)\left(\frac{a^2c^2-1}{p}\right),
\end{align*}
for odd prime $p$. With the help of our estimate of the above character sum, we have studied the tenth power mean value of generalized quadratic Gauss sums using estimates for character sums and analytic methods.
\end{abstract}
\maketitle
\section{Introduction and statements of the results}
Let $q\geq 2$ be an integer, and let $\chi$ be a Dirichlet character modulo $q$. For $n\in \mathbb{Z}$, we define the generalized quadratic Gauss sum $G(n,\chi;q)$ as
\begin{equation}
 G(n,\chi;q)=\sum_{a=1}^q\chi(a)e\left(\frac{na^2}{q}\right),\notag
\end{equation}
where $e(y)=e^{2\pi iy}$. This sum generalizes the classical quadratic Gauss sum $G(n;q)$, which is defined as
\begin{equation}
 G(n;q)=\sum_{a=1}^{q}e\left(\frac{na^2}{q}\right).\notag
\end{equation}
This kind of character sum has been studied for a long time. The values of $G(n,\chi;q)$ behave irregularly whenever $\chi$ varies. For a positive integer $n$ with $\gcd(n,q)=1$, one can find a non-trivial upper bound of $|G(n,\chi;q)|$. For such results see the work of Cochrane and Zheng \cite{cochrane}. In case of prime $p$, finding such bounds is due to Weil \cite{weil}. Let $p$ be an odd prime and $L(s,\chi)$ denote the Dirichlet $L$-function corresponding to the character $\chi \bmod p$. Let $\chi_0$ denote the principal character modulo $p$.
\par For a general integer $m\geq 3$, whether there exists an asymptotic formula for
 \begin{align*}
 \sum_{\chi\bmod p}|G(n,\chi;p)|^{2m} \text{ and }  \sum_{\chi \neq \chi_0}|G(n,\chi;p)|^{2m}|L(1,\chi)|
 \end{align*}
 is an unsolved problem. In \cite{zhang}, the third author conjectured the following.
\begin{conjecture}\label{C1}
For all positive integer $m$,
\begin{align*}
\sum_{\chi\neq\chi_0}|G(n,\chi;p)|^{2m}\cdot |L(1,\chi)|\sim C\sum_{\chi \bmod p}|G(n,\chi;p)|^{2m}, \qquad p\rightarrow+\infty,
\end{align*}
where
\begin{align}\label{constant-c}
C=\prod_p\left[1+\frac{\binom{2}{1}^2}{4^2.p^2}+\frac{\binom{4}{2}^2}{4^4.p^4}+\cdots+\frac{\binom{2m}{m}^2}{4^{2m}.p^{2m}}+\cdots\right]
\end{align}
is a constant and $\displaystyle \prod_p$ denotes the product over all primes.
\end{conjecture} Here $\displaystyle \sum_{\chi \bmod p}$ denotes the sum over all Dirichlet characters modulo $p$ and
 $\displaystyle \sum_{\chi \neq \chi_0}$ denotes the sum over all non-principal Dirichlet characters modulo $p$.
The third author \cite{zhang} showed that $G(n,\chi;p)$ satisfies many good weighted mean value properties. He used estimates for character sums and analytic methods
to study second, fourth and sixth order moments of generalized quadratic Gauss sums. To be specific, he proved that for any
integer $n$ with $\gcd(n,p)=1$
\begin{align*}
\sum_{\chi\bmod p}|G(n,\chi;p)|^4=\begin{cases}
(p-1)[3p^2-6p-1+4\left(\frac{n}{p}\right)\sqrt{p}], &\text{ if } p\equiv 1 \bmod 4;\\
(p-1)(3p^2-6p-1), &\text{ if } p\equiv 3 \bmod 4,
\end{cases}
\end{align*}
and
\begin{align*}
\sum_{\chi\bmod p}|G(n,\chi;p)|^6=(p-1)(10p^3-25p^2-4p-1), \text{ if } p\equiv 3 \bmod 4,
\end{align*}
where $\left(\frac{\bullet}{p}\right)$ is the Legendre symbol.
Later, He and Liao \cite{yuan} evaluated the sum $\displaystyle  \sum_{\chi\bmod p}|G(n,\chi;p)|^6$ when $p\equiv 1\bmod{4}$. They have also obtained the $8$-th
order mean value of generalized quadratic Gauss sums. To be specific, for any integer $n$ with $\gcd(n,p)=1$,
they \cite[Theorem 2 and 3]{yuan} proved that
\begin{align*}
\sum_{\chi\bmod p}|G(n,\chi;p)|^6=
\left\{
\begin{array}{ll}
(p-1)(10p^3-25p^2-16p-1)+(p\sqrt{p}(p-1)N\\
+18p^2\sqrt{p}-12p\sqrt{p}-6\sqrt{p})\left(\frac{n}{p}\right), & \hspace{-1.8cm} \hbox{if $p\equiv 1 \bmod 4$;} \\
(p-1)(10p^3-25p^2-4p-1), & \hspace{-1.8cm} \hbox{if $p\equiv 3 \bmod 4$,}
\end{array}
\right.
\end{align*}
where
\begin{align}\label{X1}
N=\sum_{a=2}^{p-2}\sum_{c=1}^{p-1}\left(\frac{a^2-c^2}{p}\right)\left(\frac{c^2-1}{p}\right)\left(\frac{a^2-1}{p}\right),
\end{align}
and
\begin{align*}
&\sum_{\chi\bmod p}|G(n,\chi;p)|^8\\
&=
\left\{
\begin{array}{ll}
(p-1)(34p^4-99p^3-65p^2-29p-1)\\+(56p^3\sqrt{p}+8p^2\sqrt{p}-56p\sqrt{p}-8\sqrt{p}+8p^2\sqrt{p}(p-1)N)\left(\frac{n}{p}\right)\\+p^2(p-1)T,
& \hspace{-1.8cm} \hbox{if $p\equiv 1 \bmod 4$;} \\
(p-1)(34p^4-99p^3+7p^2-5p-1)+p^2(p-1)T, & \hspace{-1.8cm} \hbox{if $p\equiv 3 \bmod 4$,}
\end{array}
\right.
\end{align*}
where $N$ is the same as \eqref{X1} and
\begin{align*}
T=\sum_{a=2}^{p-2}\sum_{b=1}^{p-1}\sum_{d=1}^{p-1}\left(\frac{a^2-b^2}{p}\right)\left(\frac{b^2-1}{p}\right)\left(\frac{a^2-d^2}{p}\right)\left(\frac{d^2-1}{p}\right).
\end{align*}
In article \cite{BB}, the first author and Barman derived asymptotic formulas for $T$ and $N$ which allow them to get an improved estimate for He and Liao's result. In particular, for odd prime $p$ and for any integer $n$ with $\gcd(n,d)=1$ they
 proved that
\begin{align*}
\sum_{\chi \bmod p} \left|G(n,\chi;p)\right|^6
=\begin{cases}
 10p^4+O(p^{3/2}),&\text{if}~p\equiv 1\bmod 4;\\
 (p-1)(10p^3-25p^2-4p-1),&\text{if}~p\equiv 3\bmod 4,
 \end{cases}
\end{align*}
and
\begin{align*}
\sum_{\chi \bmod p} \left|G(n,\chi;p)\right|^8
=35p^5+O(p^{9/2}).
 \end{align*}
 The estimates for the $6$-th and $8$-th order power mean values along with the results for $6$-th and $8$-th order power moments of generalised quadratic Gauss sums weighted by $L$-functions proved Conjecture \ref{C1} upto $m\leq 4$.
\par In this article, we first estimate a multi-variable character sum. In particular, we prove the following:
\begin{theorem}\label{MT1}
Let $p$ be an odd prime. Then we have

\begin{align*}
\sum_{a=2}^{p-2}\sum_{b=1}^{p-1}\sum_{c=2}^{p-2}\sum_{d=1}^{p-1}\left(\frac{a^2-b^2}{p}\right)\left(\frac{b^2-1}{p}\right)
  \left(\frac{c^2-d^2}{p}\right)\left(\frac{d^2-1}{p}\right)\left(\frac{a^2c^2-1}{p}\right)=O(p^2).
\end{align*}
\end{theorem}

\begin{remark}\label{remark}
For prime $p=4k+3$, one can see that the above expression is equal to zero, which can be seen by replacing $a$, $b$, $c$ and $d$ by their inverses modulo $p$.
\end{remark}
The techniques used to prove this estimate is important because one can see that the previous efforts to evaluate such character sums seems to stop at using Weil's result on curves. Our approach is a conceptual advancement to the previously known methods. With the help of the above estimate, we study $10$-th power mean value of generalized quadratic Gauss sums. To be specific, we prove the following:
\begin{theorem}\label{MT2}
Let $p$ be an odd prime and $n$ be any integer with $\gcd (n,p)=1$. Then we have
\begin{align*}
&\sum_{\chi\bmod p}|G(n,\chi;p)|^{10}=126\cdot p^6+O(p^{11/2}).\\
\end{align*}
\end{theorem}
As some applications of these results, we can also deduce the following:
\begin{theorem}\label{MT1} Let $p$ be a prime. Then for any integer $n$ with $(n, p)=1$, we have the asymptotic formula
\begin{eqnarray*}
\sum_{\chi\neq\chi_0}\left|\sum_{a=1}^{p-1}\chi(a) e\left(\frac{na^2}{p}\right)\right|^{10}\cdot |L(1,\chi)|=126\cdot C\cdot p^6+ O\left(p^{\frac{11}{2}}\cdot \ln^2 p\right),
\end{eqnarray*}
where $C$ is defined as the same as in (1.1).
\end{theorem}

\begin{theorem}\label{MT1} Let $p$ be an odd prime, $\chi$ be any non-principal even character modulo $p$. Then we have the asymptotic formula
\begin{eqnarray*}
\mathop{\sum_{\chi\bmod p}}_
{\chi\neq\chi_0}\left|\sum_{a=1}^{p-1}\chi\left(a+\overline{a}\right) \right|^{4}=3 \cdot p^3+ O\left(p^{\frac{5}{2}}\cdot \ln p\right),
\end{eqnarray*}
where $\overline{a}$ denotes the multiplicative inverse of $a$. That is, $a\cdot \overline{a}\equiv 1\bmod p$.
\end{theorem}

\begin{theorem}\label{MT1}
 Let $p$ be an odd prime. Then we have the asymptotic formula
\begin{eqnarray*}
\mathop{\sum_{\chi\bmod p}}_
{\chi\neq\chi_0}\left|\sum_{a=1}^{p-1}\chi\left(a+\overline{a}\right) \right|^{4}\cdot |L(1,\chi)|=3\cdot C \cdot p^3 + O\left(p^{\frac{5}{2}}\cdot \ln^2 p\right).
\end{eqnarray*}
\end{theorem}

In fact for any positive integer $k$, we have the following several conjectures:

\begin{conjecture}\label{MT1} Let $p$ be a prime large enough, $k$ be any positive integer. Then we have the asymptotic formula
\begin{eqnarray*}
\frac{1}{p^{k+1}}\cdot \sum_{\chi \bmod p}\left|\sum_{a=1}^{p-1}\chi(a) e\left(\frac{na^2}{p}\right)\right|^{2k}=\binom{2k-1}{k} + o(1).
\end{eqnarray*}
\end{conjecture}

\begin{conjecture}\label{MT1} Let $p$ be a prime large enough, $k$ be any positive integer. Then for any integer $n$ with $(n, p)=1$, we have the asymptotic formula
\begin{eqnarray*}
\frac{1}{p^{k+1}}\cdot\sum_{\chi\neq\chi_0}\left|\sum_{a=1}^{p-1}\chi(a) e\left(\frac{na^2}{p}\right)\right|^{2k}\cdot |L(1,\chi)|=\binom{2k-1}{k}\cdot C+ o\left(1\right).
\end{eqnarray*}
\end{conjecture}

It is clear that Conjecture 1.8  and Conjecture 1.9 are correct for $k=1, \ 2, \ 3, \ 4, \ 5$. For integer $k\geq 6$, whether they are correct are two open problems.

For Theorem 1.6 and Theorem 1.7, we also have two  corresponding conjectures:

\begin{conjecture}\label{MT1}  Let $p$ be a prime large enough. Then for any positive integer $k$, we have the asymptotic formula
\begin{eqnarray*}
\frac{1}{p^{k+1}}\cdot \mathop{\sum_{\chi \bmod p}}_{\chi\neq \chi_0}\left|\sum_{a=1}^{p-1}\chi\left(a+\overline{a}\right)\right|^{2k}=\binom{2k-1}{k} + o(1).
\end{eqnarray*}
\end{conjecture}

\begin{conjecture}\label{MT1}  Let $p$ be a prime large enough. Then for any positive integer $k$, we have the asymptotic formula
\begin{eqnarray*}
\frac{1}{p^{k+1}}\cdot \mathop{\sum_{\chi\bmod p}}_
{\chi\neq\chi_0}\left|\sum_{a=1}^{p-1}\chi\left(a+\overline{a}\right) \right|^{2k}\cdot |L(1,\chi)|= \binom{2k-1}{k}\cdot C+ o\left(1\right).
\end{eqnarray*}
\end{conjecture}

\section{Some lemmas }
For convenience, throughout the paper we denote $A=1+\chi(-1)$ and $B=\left(\frac{n}{p}\right)G(1;p)$,
where $G(1;p)$ is the Gauss sum $\displaystyle G(1;p)=\sum_{b=0}^{p-1}e\left(\frac{b^2}{p}\right)$ and $\left(\frac{\bullet}{p}\right)$ is the Legendre symbol. We need the following lemmas to prove our main result.
\begin{lemma}[Lemma 1 \cite{yuan}]\label{X7}
 Let $p$ be an odd prime. Then
  \begin{align*}\label{X7}
 \displaystyle \sum_{a=2}^{p-2}\left(\frac{a^2-1}{p}\right)=\begin{cases}
                                                -2, &p\equiv 1\bmod 4;\\
                                                0, &p\equiv 3\bmod 4.
                                               \end{cases}
\end{align*}
\end{lemma}
\begin{lemma}[Theorem 1 \cite{yuan}]\label{X6}
 Let $p$ be an odd prime. Then
 \begin{align*}
 \displaystyle \sum_{a=2}^{p-2}\sum_{b=1}^{p-1}\left(\frac{a^2-b^2}{p}\right)\left(\frac{b^2-1}{p}\right)=\begin{cases}
                                                                                              10-2p, &p\equiv 1\bmod4;\\
                                                                                              2p-6,&p\equiv 3\bmod 4.
                                                                                             \end{cases}
\end{align*}
\end{lemma}
\begin{lemma}[Lemma 2 \cite{yuan}]\label{X2}
 Let $p$ be an odd prime and $n$ be any integer with $\gcd(n,p)=1$. Then for any non-principal character  $\chi$ modulo $p$ the following identity holds
 \begin{equation*}
  |G(n,\chi;p)|^2=Ap+B\sum_{a=2}^{p-2}\chi(a)\left(\frac{a^2-1}{p}\right).
 \end{equation*}
 If $\chi_0$ is the principal character modulo $p$, then
 \begin{align*}
 |G(n,\chi_0;p)|^2=\begin{cases}
 p+1-2\sqrt{p}\left(\frac{n}{p}\right), &p\equiv 1\bmod 4;\\
 p+1, &p\equiv 3\bmod 4.
 \end{cases}
 \end{align*}
\end{lemma}
\begin{lemma}[Lemma 3 \cite{yuan}]\label{X3}
Let $p$ be an odd prime, and let $\chi_0$ be the principal character modulo $p$. Then for any positive integer $m$, we have
\begin{align*}
\displaystyle\sum_{\chi\neq\chi_0}(1+\chi(-1))^m=\begin{cases}\displaystyle
(p-3)\sum_{i=0}^{\frac{m}{2}}\binom{m}{2i}, \text{ if $m$ is an even integer;} \\
\displaystyle(p-3)\sum_{i=0}^{\frac{m-1}{2}}\binom{m}{2i}, \text{ if $m$ is an odd integer},
\end{cases}
\end{align*}
here $\binom{m}{2i}=\frac{m!}{(2i)!\cdot(m-2i)!}$.
\end{lemma}
\begin{lemma}[Lemma 4 \cite{yuan}]\label{X4}
 Let $p$ be an odd prime, and let $\chi_0$ be the principal character modulo $p$. Then for any positive integre $m$, we have
 \begin{align*}
  \sum_{\chi\neq\chi_0}(1+\chi(-1))^m\sum_{a=2}^{p-2}\chi(a)\left(\frac{a^2-1}{p}\right)=\begin{cases}
                                                                                          2^{m+1}, &p\equiv 1\bmod 4;\\
                                                                                          0,  &p\equiv 3\bmod 4.
                                                                                         \end{cases}
 \end{align*}
\end{lemma}
The next important result is due to Gauss.
\begin{lemma}[Section 9.10 \cite{gauss}]\label{X11}
 For any integer $q\geq1$, we have
 \begin{equation*}
  G(1;q)=\frac{1}{2}\sqrt{q}(1+i)(1+e^{\frac{-\pi iq}{2}})=\begin{cases}\sqrt{q} &\ \text{if}\quad  q\equiv1\bmod 4;\\
0 &\ \text{if}\quad  q\equiv2\bmod 4;\\
i\sqrt{q} &\ \text{if}\quad  q\equiv3\bmod 4;\\
(1+i)\sqrt{q}  &\ \text{if}\quad

 q\equiv0\bmod 4.
\end{cases}
  \end{equation*}
\end{lemma}

\begin{lemma}\cite[Theorem 1.4]{BB}\label{sum-2} Let $p$ be an odd prime. Then we have
	\begin{align*}
	\sum_{b=1}^{p-1}\sum_{c=2}^{p-2}\left(\frac{b^2-c^2}{p}\right)\left(\frac{b^2-1}{p}\right)\left(\frac{c^2-1}{p}\right)=O(p).
	\end{align*}
\end{lemma}

\begin{lemma}\label{lem-T}   Let $p$ be an odd prime, $\chi$ be any non-principal
character $\bmod\ p$. Then for any integer $m$ with $(m, p)=1$, we
have the identity
\begin{eqnarray*}
\left|\sum_{a=1}^{p-1}\chi\left(ma+\overline{a}\right)\right|=\left|\sum_{a=1}^{p-1}\chi(a)\left(\frac{a^2-m}{p}\right)\right|.
\end{eqnarray*}
\end{lemma}

\begin{proof} It is clear that if $\chi$ is an odd character modulo $p$, then both sides of the lemma are zero. So without loss of generality we can assume that $\chi$ is a non-principal even character modulo $p$.
  Let $am+\overline{a}=u$, then from the definition of $\overline{a}$ and the properties of the congruence modulo $p$
 we know that for any $(m, p)=1$, we have
\begin{eqnarray}
&&\sum_{a=1}^{p-1}\chi\left(ma+\overline{a}\right)=\sum_{u=1}^{p-1}\chi(u)\mathop{\sum_{a=1}^{p-1}}_{am+\overline{a}\equiv
u\bmod p } 1\nonumber\\
&=&\sum_{u=1}^{p-1}\chi(u)\mathop{\sum_{a=1}^{p-1}}_{a^2m^2-amu+m
\equiv0\bmod p } 1\nonumber\\
&=&\sum_{u=1}^{p-1}\chi(u)\mathop{\sum_{a=0}^{p-1}}_{(2am-u)^2\equiv
u^2-4m\bmod p }
1=\sum_{u=1}^{p-1}\chi(u)\mathop{\sum_{a=0}^{p-1}}_{a^2\equiv
u^2-4m\bmod p } 1.
\end{eqnarray}
Note that for any fixed integer $u^2-4m$, the number of the
solutions of the congruent equation $x^2\equiv u^2-4m \bmod p$ are
$1+ \left(\frac{u^2-4m}{p}\right)$, so from (2.1) we have
\begin{eqnarray*}
\sum_{a=1}^{p-1}\chi\left(ma+\overline{a}\right)=\sum_{u=1}^{p-1}\chi(u)
\left(1+\left(\frac{u^2-4m}{p}\right)\right)=\chi(2)\sum_{u=1}^{p-1}\chi(u)
\left(\frac{u^2-m}{p}\right),
\end{eqnarray*}
which implies Lemma 2.8.
\end{proof}

\section{The proofs of the main theorems}

In this section,  we prove our estimate for a multi-variable character sums, where we relate our sum to trace of Frobenius of $\ell$-adic sheaves.
\begin{proof}[Proof of Theorem 1.2]
We write $\rho(t)=\left(\frac{t}{p}\right)$ and denote the our sum as $S$. Hence we have
\begin{align*}
S&=\sum_{a=2}^{p-2}\sum_{b=1}^{p-1}\sum_{c=2}^{p-2}\sum_{d=1}^{p-1}\rho(a^2-b^2)\rho(b^2-1)\rho(c^2-d^2)\rho(d^2-1)\rho(a^2c^2-1)\\
&=\sum_{a=2}^{p-2}\sum_{b=1}^{p-1}\sum_{c=2}^{p-2}\sum_{d=1}^{p-1}\rho(a^2-b^2)\rho(b^2-1)\rho(c^2-d)\rho(d-1)\rho(a^2c^2-1)(1+\rho(d))\\
&=S_1+S_2,
\end{align*}
where
\begin{align*}
S_1&=\sum_{a=2}^{p-2}\sum_{b=1}^{p-1}\sum_{c=2}^{p-2}\sum_{d=1}^{p-1}\rho(a^2-b^2)\rho(b^2-1)\rho(c^2-d)\rho(d-1)\rho(a^2c^2-1),\\
S_2&=\sum_{a=2}^{p-2}\sum_{b=1}^{p-1}\sum_{c=2}^{p-2}\sum_{d=1}^{p-1}\rho(a^2-b^2)\rho(b^2-1)\rho(c^2-d)\rho(d-1)\rho(a^2c^2-1)\rho(d).
\end{align*}
Now
\begin{align*}
S_1&=\sum_{a=2}^{p-2}\sum_{b=1}^{p-1}\sum_{c=2}^{p-2}\rho(a^2-b^2)\rho(b^2-1)\rho(a^2c^2-1)\left(\frac{-1}{p}\right)\sum_{d=1}^{p-1}\rho(d-c^2)\rho(d-1)\\
&=\sum_{a=2}^{p-2}\sum_{b=1}^{p-1}\sum_{c=2}^{p-2}\rho(a^2-b^2)\rho(b^2-1)\rho(a^2c^2-1)\left(\frac{-1}{p}\right)\sum_{d=0}^{p-1}\rho(d-c^2)\rho(d-1)\\
&
\qquad\qquad -\sum_{a=2}^{p-2}\sum_{b=1}^{p-1}\sum_{c=2}^{p-2}\rho(a^2-b^2)\rho(b^2-1)\rho(a^2c^2-1)\left(\frac{-1}{p}\right).
\end{align*}
The inner sum in the first expression is $-1$ as $c\neq \pm 1$. Hence we have
\begin{align*}
S_1=-2\left(\frac{-1}{p}\right)\sum_{a=2}^{p-2}\sum_{b=1}^{p-1}\sum_{c=2}^{p-2}\rho(a^2-b^2)\rho(b^2-1)\rho(a^2c^2-1)=O(p^2)
\end{align*}
since
\begin{align*}
\sum_{c=2}^{p-2}\rho(a^2c^2-1)=-2\rho(a^2-1)+\sum_{c=1}^{p-1}\rho(c^2-1)=-2\rho(a^2-1)-1-\rho(-1).
\end{align*}
Hence we have
\begin{align*}
S=S_2+O(p^2).
\end{align*}
A similar reduction for $b$ on $S_2$ will give
\begin{align}\label{TH1}
S=S_3+O(p^2),
\end{align}
where
\begin{align*}
S_3=\sum_{a=2}^{p-2}\sum_{b=1}^{p-1}\sum_{c=2}^{p-2}\sum_{d=1}^{p-1}\rho(a^2-b)\rho(b-1)\rho(c^2-d)\rho(d-1)\rho(a^2c^2-1)\rho(b)\rho(d).
\end{align*}
We now follow the same technique on $S_3$. Notice that
\begin{align}\label{TH2}
S_3&=\sum_{a=1}^{p-1}\sum_{b=1}^{p-1}\sum_{c=2}^{p-2}\sum_{d=1}^{p-1}\rho(a^2-b)\rho(b-1)\rho(c^2-d)\rho(d-1)\rho(a^2c^2-1)\rho(b)\rho(d)+O(p^2)\notag\\
&=\sum_{a=1}^{p-1}\sum_{b=1}^{p-1}\sum_{c=2}^{p-2}\sum_{d=1}^{p-1}\rho(a-b)\rho(b-1)\rho(c^2-d)\rho(d-1)\rho(ac^2-1)\rho(b)\rho(d)(1+\rho(a))\notag\\&\hspace{10cm}+O(p^2)\notag\\
&=U_1+U_2+O(p^2),
\end{align}
where
\begin{align*}
U_1=\sum_{a=1}^{p-1}\sum_{b=1}^{p-1}\sum_{c=2}^{p-2}\sum_{d=1}^{p-1}\rho(a-b)\rho(b-1)\rho(c^2-d)\rho(d-1)\rho(ac^2-1)\rho(b)\rho(d)
\end{align*}
and
\begin{align*}
U_2=\sum_{a=1}^{p-1}\sum_{b=1}^{p-1}\sum_{c=2}^{p-2}\sum_{d=1}^{p-1}\rho(a-b)\rho(b-1)\rho(c^2-d)\rho(d-1)\rho(ac^2-1)\rho(b)\rho(d)\rho(a).
\end{align*}
Now
\begin{align*}
U_1&=\sum_{b=1}^{p-1}\sum_{c=2}^{p-2}\sum_{d=1}^{p-1}\rho(b-1)\rho(c^2-d)\rho(d-1)\rho(b)\rho(d)\sum_{a=1}^{p-1}\rho(ac^2-1)\rho(a-b)\\
&=\sum_{b=1}^{p-1}\sum_{c=2}^{p-2}\sum_{d=1}^{p-1}\rho(b-1)\rho(c^2-d)\rho(d-1)\rho(b)\rho(d)\sum_{a=0}^{p-1}\rho(ac^2-1)\rho(a-b)\\
&\qquad\qquad-\sum_{b=1}^{p-1}\sum_{c=2}^{p-2}\sum_{d=1}^{p-1}\rho(b-1)\rho(c^2-d)\rho(d-1)\rho(d)\\
&=\sum_{b=1}^{p-1}\sum_{c=2}^{p-2}\sum_{d=1}^{p-1}\rho(b-1)\rho(c^2-d)\rho(d-1)\rho(b)\rho(d)\sum_{a=0}^{p-1}\rho(ac^2-1)\rho(a-b)+O(1)\\
&=-\sum_{b=1}^{p-1}\sum_{c=2, b\neq c^{-2}}^{p-2}\sum_{d=1}^{p-1}\rho(b-1)\rho(c^2-d)\rho(d-1)\rho(b)\rho(d)\\
&\qquad\qquad\qquad\qquad+(p-1)\sum_{c=2}^{p-2}\sum_{d=1}^{p-1}\rho(c^{-2}-1)\rho(c^2-d)\rho(d-1)\rho(d)+O(1),
\end{align*}
since
\begin{align*}
-\sum_{b=1}^{p-1}\sum_{c=2}^{p-2}\sum_{d=1}^{p-1}\rho(b-1)\rho(c^2-d)\rho(d-1)\rho(d) = \rho(-1)\sum_{c=2}^{p-2}\sum_{d=1}^{p-1}\rho(c^2-d)\rho(d-1)\rho(d) \\
= \rho(-1)\sum_{d=1}^{p-1}(-1-\rho(-d)-2\rho(1-d))\rho(d-1)\rho(d)=2+2\rho(-1).
\end{align*}
Consider the second sum,
\begin{align}
&\sum_{c=2}^{p-2}\sum_{d=1}^{p-1}\rho(c^{-2}-1)\rho(c^2-d)\rho(d-1)\rho(d)\notag\\
&=\sum_{c=1}^{p-1}\sum_{d=1}^{p-1}\rho(c^{-2}-1)\rho(c^2-d)\rho(d-1)\rho(d)\notag\\
&=\left(\frac{-1}{p}\right)\sum_{c=1}^{p-1}\sum_{d=1}^{p-1}\rho(c-1)\rho(c-d)\rho(d-1)\rho(d)(1+\rho(c))\notag\\
&=1+\left(\frac{-1}{p}\right)+\left(\frac{-1}{p}\right)\sum_{c=1}^{p-1}\sum_{d=1}^{p-1}\rho(c-1)\rho(c-d)\rho(d-1)\rho(d)\rho(c)\notag\\
&=1+\left(\frac{-1}{p}\right)+\sum_{c=1}^{p-1}\rho(c^2-c)\phi(c)\notag,
\end{align}
where
\begin{align*}
\phi(c)=\sum_{d=1}^{p-1}\rho(d-c)\rho(d-1)\rho(d)=\sum_{d=0}^{p-1}\rho(d-c)\rho(d-1)\rho(d).
\end{align*}
In order to estimate this sum, we will interpret it as the Frobenius trace function associated to an $\ell$-adic sheaf $\mathcal F$ for some fixed prime $\ell\neq p$. See \cite{FKMS} for an overview of the involved theory. More precisely, $\mathcal F$ will be the first cohomology sheaf of the Legendre family of elliptic curves
 \begin{align*}
 E(t): y^2=x(x-1)(x-t)
 \end{align*}
which is constructed as follows: let $S={\mathbb A}^1_k$ be the affine line (with coordinate $t$) over the finite field $k:={\mathbb F}_p$, $X\subseteq {\mathbb A}^2_S$ the elliptic curve with equation $y^2=x(x-1)(x-t)$, and ${\mathcal F}={\mathrm R}^1\pi_!{\mathbb Q_\ell}$, where $\pi:X\to S$ is the structural map.

It is known (see eg. \cite[10.1]{katz-sarnak}) that $\mathcal F$ is a rank 2 smooth sheaf on $S-\{0,1\}$. At these two points it has unipotent local monodromy, and the local monodromy at infinity is unipotent tensored with an order 2 character. The action of a geometric Frobenius element at $t\in k$ has trace $p-N(t)=-\phi(t)$, where $N(t)$ is the number of $k$-rational points on the curve $E(t)$.

By the Grothendieck-Lefschetz trace formula, we have
$$
\sum_{c=1}^{p-1}\rho(c^2-c)\phi(c)=\mathrm{Tr}(Fr|{\mathrm H}^1_c({\mathbb A}^1_{\bar k},{\mathcal F}\otimes{\mathcal L}))-\mathrm{Tr}(Fr|{\mathrm H}^2_c({\mathbb A}^1_{\bar k},{\mathcal F}\otimes{\mathcal L}))
$$
where $\mathcal L=[t\mapsto t^2-t]^\ast{\mathcal L}_\rho$ is the rank one (pull-back of) Kummer sheaf whose Frobenius trace function at $c\in k$ is $\rho(c^2-c)$. Since $\mathcal L$ has order two monodromy action at $t=0$, the monodromy action at $0$ on the tensor product ${\mathcal F}\otimes{\mathcal L}$ is unipotent tensored with an order 2 character. In particular, ${\mathcal F}\otimes{\mathcal L}$ does not have any geometrically constant component (which would have trivial local monodromy at $0$), which implies that ${\mathrm H}^2_c({\mathbb A}^1_{\bar k},{\mathcal F}\otimes{\mathcal L})=0$. Therefore,
$$
\sum_{c=1}^{p-1}\rho(c^2-c)\phi(c)=\mathrm{Tr}(Fr|{\mathrm H}^1_c({\mathbb A}^1_{\bar k},{\mathcal F}\otimes{\mathcal L}))
$$
and, since $\mathcal F$ (and therefore ${\mathcal F}\otimes{\mathcal L}$) is mixed of weight $\leq 1$, all Frobenius eigenvalues of ${\mathrm H}^1_c({\mathbb A}^1_{\bar k},{\mathcal F}\otimes{\mathcal L})$ have absolute value $\leq p$. Moreover, the Ogg-Shafarevic formula implies that $\dim({\mathrm H}^1_c({\mathbb A}^1_{\bar k},{\mathcal F}\otimes{\mathcal L}))=-\chi({\mathcal F}\otimes{\mathcal L})=2$ (since ${\mathcal F}\otimes{\mathcal L}$ is smooth of rank $2$ on $S-\{0,1\}$, tamely ramified everywhere, and has rank 0 at $t=1$), so we get the estimate
\begin{align}\label{TH3}
\left|\sum_{c=1}^{p-1}\rho(c^2-c)\phi(c)\right|\leq 2p.
\end{align}

Now consider
\begin{align}
&\sum_{c=2}^{p-2}\sum_{b=1, b\neq c^{-2}}^{p-1}\sum_{d=1}^{p-1}\rho(b-1)\rho(c^2-d)\rho(d-1)\rho(b)\rho(d)\notag\\
&=\sum_{c=2}^{p-2}\sum_{b=1}^{p-1}\sum_{d=1}^{p-1}\rho(b-1)\rho(c^2-d)\rho(d-1)\rho(b)\rho(d)\notag\\
&\qquad\qquad-\left(\frac{1}{p}\right)\sum_{c=2}^{p-2}\sum_{d=1}^{p-1}\rho(c^{2}-1)\rho(c^2-d)\rho(d-1)\rho(d)\notag\\
&=O(p^2)\notag.
\end{align}
Hence
\begin{align}\label{TH4}
U_1=O(p^2).
\end{align}
Now rewriting $U_2$ and using \eqref{TH3} we get
\begin{align*}
U_2&=\sum_{a=1}^{p-1}\sum_{b=1}^{p-1}\sum_{c=2}^{p-2}\sum_{d=1}^{p-1}\rho(a-b)\rho(b-1)\rho(c^2-d)\rho(d-1)\rho(ac^2-1)\rho(b)\rho(d)\rho(a)\notag\\
&=\sum_{a=1}^{p-1}\sum_{b=1}^{p-1}\sum_{c=1}^{p-1}\sum_{d=1}^{p-1}\rho(a-b)\rho(b-1)\rho(c-d)\rho(d-1)\rho(ac-1)\rho(b)\rho(d)\rho(a)(1+\rho(c))\\
&\qquad+2\rho(-1)\sum_{b=1}^{p-1}\rho(b)\rho(b-1)\phi(b)\notag\\
&=\sum_{a=1}^{p-1}\sum_{b=1}^{p-1}\sum_{c=1}^{p-1}\sum_{d=1}^{p-1}\rho(a-b)\rho(b-1)\rho(c-d)\rho(d-1)\rho(ac-1)\rho(b)\rho(d)\rho(a)(1+\rho(c))+O(p)\notag\\
&=U_2^{'}+U_2^{''}+O(p),
\end{align*}
where
\begin{align*}
&U_2^{'}=\sum_{a=1}^{p-1}\sum_{b=1}^{p-1}\sum_{c=1}^{p-1}\sum_{d=1}^{p-1}\rho(a-b)\rho(b-1)\rho(c-d)\rho(d-1)\rho(ac-1)\rho(b)\rho(d)\rho(a)\\
&U_2^{''}=\sum_{a=1}^{p-1}\sum_{b=1}^{p-1}\sum_{c=1}^{p-1}\sum_{d=1}^{p-1}\rho(a-b)\rho(b-1)\rho(c-d)\rho(d-1)\rho(ac-1)\rho(b)\rho(d)\rho(a)\rho(c).
\end{align*}
Proceeding similar as before we have
\begin{align*}
U_2^{'}=\sum_{a=1}^{p-1}\sum_{b=1}^{p-1}\sum_{d=1}^{p-1}\rho(a-b)\rho(b-1)\rho(d-1)\rho(b)\rho(d)\sum_{c=1}^{p-1}\rho(c-d)\rho(c-a^{-1})\\
=-\rho(-1)+\sum_{a=1}^{p-1}\sum_{b=1}^{p-1}\sum_{d=1}^{p-1}\rho(a-b)\rho(b-1)\rho(d-1)\rho(b)\rho(d)\\
\sum_{c=0}^{p-1}\rho(c-d)\rho(c-a^{-1})\\
=-\rho(-1)-\sum_{a=1}^{p-1}\sum_{b=1}^{p-1}\sum_{d=1,d\neq a^{-1}}^{p-1}\rho(a-b)\rho(b-1)\rho(d-1)\rho(b)\rho(d)\\
\qquad\qquad +(p-1)\sum_{a=1}^{p-1}\sum_{b=1}^{p-1}\rho(a-b)\rho(b-1)\rho(a^{-1}-1)\rho(b)\rho(a^{-1})\\
=\rho(-1)(p-1)\sum_{b=1}^{p-1}\rho(b)\rho(b-1)\sum_{a=1}^{p-1}\rho(a-b)\rho(a-1)+O(p^2)\\
=\rho(-1)(p-1)\sum_{b=2}^{p-1}\rho(b)\rho(b-1)\sum_{a=0}^{p-1}\rho(a-b)\rho(a-1)+O(p^2)
=O(p^2).
\end{align*}

Finally, we can write $U_{2}^{''}$ as
\begin{align*}
U_2^{''}&=\sum_{a=1}^{p-1}\sum_{c=1}^{p-1}\phi(a)\phi(c)\rho(ac-1)\rho(ac)\\
&=\sum_{a=1}^{p-1}\rho(a-1)\rho(a)\sum_{c=1}^{p-1}\phi(ac^{-1})\phi(c)=\sum_{a=1}^{p-1}\rho(a-1)\rho(a)\psi(a),
\end{align*}
where
$$
\psi(a)=\sum_{c=1}^{p-1}\phi(ac^{-1})\phi(c).
$$
This $\psi$ is the Frobenius trace function at $a\in k^\times$ of the multiplicative $!$-convolution ${\mathcal F}\ast{\mathcal F}$ (see eg. \cite[8.1]{katz-esde}). In general, the convolution is only defined in the derived category of $\ell$-adic sheaves, but in this case we claim that ${\mathcal F}\ast{\mathcal F}$ collapses to a single sheaf.

Let ${\mathbb G}_{m,k}={\mathbb A}^1_k-\{0\}$ be the one-dimensional torus over $k$. The $!$-convolution ${\mathcal F}\ast{\mathcal F}$ is defined as ${\mathrm R}\mu_!(\pi_1^\ast{\mathcal F}\otimes\pi_2^\ast{\mathcal F})$, where $\mu,\pi_1,\pi_2:{\mathbb G}_{m,k}\times{\mathbb G}_{m,k}\to{\mathbb G}_{m,k}$ are the multiplication map and the projections onto each factor. For $i\geq 0$ and $t\in\bar k^\times$, the fibre of ${\mathrm R}^i\mu_!(\pi_1^\ast{\mathcal F}\otimes\pi_2^\ast{\mathcal F})$ at $t$ is $\mathrm H^i_c({\mathbb G}_{m,\bar k},{\mathcal F}\otimes[s\mapsto t/s]^\ast{\mathcal F})$. Since $\mathcal F$ does not have punctual sections, this fibre is 0 for $i\neq 1,2$. Moreover, we know that $\mathcal F$ has unipotent local monodromy at $0$ and its monodromy at $\infty$ is unipotent tensored with an order 2 character. The same holds then for the monodromy at $0$ of $[s\mapsto t/s]^\ast{\mathcal F}$. Therefore, the monodromy at $0$ of ${\mathcal F}\otimes[s\mapsto t/s]^\ast{\mathcal F}$ is unipotent tensored with an order 2 character. In particular, ${\mathcal F}\otimes[s\mapsto t/s]^\ast{\mathcal F}$ can not have a geometrically constant component, so $\mathrm H^2_c({\mathbb G}_{m,\bar k},{\mathcal F}\otimes[s\mapsto t/s]^\ast{\mathcal F})=0$. We conclude that the cohomology of ${\mathcal F}\ast{\mathcal F}$ is concentrated in degree $1$. Let us denote ${\mathcal G}={\mathrm R}^1\mu_!(\pi_1^\ast{\mathcal F}\otimes\pi_2^\ast{\mathcal F})$. Then the Frobenius trace function of ${\mathcal G}$ at $t\in k^\times$ is $-\psi(k)$.

Since ${\mathcal F}$ has weight $\leq 1$, ${\mathcal G}$ has weight $\leq 3$. By \cite[Corollary 24]{RL}, it is tamely ramified everywhere, and its monodromy at $0$ is unipotent \cite[Proposition 28]{RL}. Moreover, it is smooth of rank $4$ on ${\mathbb G}_{m,k}-\{1\}$ by the Ogg-Shafarevic formula.

The sum $\sum_{a=1}^{p-1}\rho(a-1)\rho(a)\psi(a)$ is then minus the sum of the Frobenius traces of the sheaf ${\mathcal G}\otimes{\mathcal L}$ at the points of $k^\times$. So, by the Grothendieck-Lefschetz trace formula, it can be written as
$$
\mathrm{Tr}(Fr|{\mathrm H}^1_c({\mathbb G}_{m,\bar k},{\mathcal G}\otimes{\mathcal L}))-\mathrm{Tr}(Fr|{\mathrm H}^2_c({\mathbb G}_{m,\bar k},{\mathcal G}\otimes{\mathcal L}))
$$
where $\mathcal L=[t\mapsto t^2-t]^\ast{\mathcal L}_\rho$ is the same as above. Since $\mathcal G$ has unipotent monodromy at $0$, the monodromy at $0$ of ${\mathcal G}\otimes{\mathcal L}$ is unipotent tensored with a rank 2 character, so in particular ${\mathcal G}\otimes{\mathcal L}$ can not have any geometrically constant component and the ${\mathrm H}^2_c$ term vanishes. Since ${\mathcal G}\otimes{\mathcal L}$ has weight $\leq 3$, the Frobenius eigenvalues of its ${\mathrm H}^1_c$ have absolute value $\leq p^2$. Moreover, by the Ogg-Shafarevic formula, we have $\dim({\mathrm H}^1_c({\mathbb G}_{m,\bar k},{\mathcal G}\otimes{\mathcal L}))=-\chi({\mathrm H}^1_c({\mathbb G}_{m,\bar k},{\mathcal G}\otimes{\mathcal L}))=4$, so we get the estimate
$$
\left|\sum_{a=1}^{p-1}\rho(a-1)\rho(a)\psi(a)\right|\leq 4p^2.
$$
Hence we get 
\begin{align}\label{TH5}
U_2=O(p^2).
\end{align}
Combining \eqref{TH1}, \eqref{TH2}, \eqref{TH4} and \eqref{TH5} we prove our result.
\end{proof}

\begin{proof}[Proof of Theorem  1.4]
\par For any $n$ with $\gcd(n,p)=1$, we have
 \begin{equation}\label{X20}
   \sum_{\chi\bmod p}|G(n,\chi;p)|^{10}= \sum_{\substack{\chi\neq\chi_0}}|G(n,\chi;p)|^{10}+|G(n,{\chi}_{0};p)|^{10},
 \end{equation}
 where using Lemma \ref{X2}, we obtain
 \begin{align}\label{X21}
 |G(n,{\chi}_{0};p)|^{10}=\begin{cases}
 \left( p+1-2\sqrt{p}\left(\frac{n}{p}\right)\right)^5, &p\equiv 1\bmod 4;\\
 (p+1)^5, &p\equiv 3\bmod 4.
 \end{cases}
 \end{align}
It follows from Lemma \ref{X2} that
\begin{align}\label{X50}
   &\sum_{\chi\neq\chi_0}|G(n,\chi;p)|^{10}\\\notag
   &=\sum_{\chi\neq\chi_0}\left(Ap+B\sum_{a=2}^{p-2}\chi(a)\left(\frac{a^2-1}{p}\right)\right)^5\\\notag
   &= N_1+N_2+N_3+N_4+N_5+N_6,\notag
\end{align}
where
\begin{align*}
&N_1=\sum_{\chi\neq\chi_0}(Ap)^5;\\
&N_2=\sum_{\chi\neq\chi_0}\binom{5}{1}(Ap)^4\cdot B\sum_{a=2}^{p-2}\chi(a)\left(\frac{a^2-1}{p}\right);\\
&N_3=\sum_{\chi\neq\chi_0}\binom{5}{2}(Ap)^3\cdot \left(B\sum_{a=2}^{p-2}\chi(a)\left(\frac{a^2-1}{p}\right)\right)^2;\\
&N_4=\sum_{\chi\neq\chi_0}\binom{5}{3}(Ap)^2\cdot\left(B\sum_{a=2}^{p-2}\chi(a)\left(\frac{a^2-1}{p}\right)\right)^3;\\
&N_5=\sum_{\chi\neq\chi_0}\binom{5}{4}(Ap)\cdot\left(B\sum_{a=2}^{p-2}\chi(a)\left(\frac{a^2-1}{p}\right)\right)^4;\\
&N_6=\sum_{\chi\neq\chi_0}\binom{5}{5}\left(B\sum_{a=2}^{p-2}\chi(a)\left(\frac{a^2-1}{p}\right)\right)^5.
\end{align*}
We will evaluate $N_1, N_2, N_3, N_4, N_5$ and $N_6$ one by one. Using Lemma \ref{X3}, Lemma \ref{X4} and Lemma \ref{X11} , we get
\begin{equation}\label{X22}
 N_1=16p^5(p-3)
\end{equation}
and
\begin{align}\label{X23}
N_2=\begin{cases}160p^4\sqrt{p}\left(\frac{n}{p}\right),&p\equiv 1\bmod 4;\\
0,&p\equiv 3\bmod 4.
\end{cases}
 \end{align}
 We have the identity
\begin{align}\label{X5}
 \sum_{a=1}^{p-1}&\sum_{b=1}^{p-1}\chi(ab)\left(\frac{a^2-1}{p}\right)\left(\frac{b^2-1}{p}\right)\\ \notag
 &=\sum_{a=1}^{p-1}\sum_{b=1}^{p-1}\chi(a)\left(\frac{a^2\overline{b}^2-1}{p}\right)\left(\frac{b^2-1}{p}\right)\\ \notag
 &=\sum_{a=1}^{p-1}\sum_{b=1}^{p-1}\chi(a)\left(\frac{a^2-b^2}{p}\right)\left(\frac{b^2-1}{p}\right)\\ \notag
 &=\left(\frac{-1}{p}\right)(p-3)A+\sum_{a=2}^{p-2}\sum_{b=1}^{p-1}\chi(a)\left(\frac{a^2-b^2}{p}\right)\left(\frac{b^2-1}{p}\right).
 \end{align}
 Using  \eqref{X5}, we obtain
 \allowdisplaybreaks
 \begin{align*}
  &\sum_{\chi\neq\chi_0}\binom{5}{2}(Ap)^3\cdot\left(B\sum_{a=2}^{p-2}\chi(a)\left(\frac{a^2-1}{p}\right)\right)^2\\
  &=40p^3B^2\sum_{\chi\neq\chi_0}A\cdot\left(\sum_{a=2}^{p-2}\chi(a)\left(\frac{a^2-1}{p}\right)\right)^2\\
  &=40p^3B^2\sum_{\chi\neq\chi_0}A\cdot\left(\sum_{a=2}^{p-2}\chi(a)\left(\frac{a^2-1}{p}\right)\right)^2\\
  &=40p^3B^2\sum_{\chi\neq\chi_0}A\left[\left(\frac{-1}{p}\right)(p-3)A+\sum_{a=2}^{p-2}\sum_{b=1}^{p-1}\chi(a)\left(\frac{a^2-b^2}{p}\right)
  \left(\frac{b^2-1}{p}\right)\right]\\
    &=40p^3B^2\left[\left(\frac{-1}{p}\right)(p-3)\sum_{\chi\neq\chi_0}A^2+\sum_{a=2}^{p-2}\sum_{b=1}^{p-1}\left(\frac{a^2-b^2}{p}\right)
  \left(\frac{b^2-1}{p}\right)\left(\sum_{\chi\bmod p}\chi(a)-1\right)\right.\\&+\left.\sum_{a=2}^{p-2}\sum_{b=1}^{p-1}\left(\frac{a^2-b^2}{p}\right)
  \left(\frac{b^2-1}{p}\right)\left(\sum_{\chi\bmod p}\chi(-a)-1\right)\right].
  \end{align*}
  Thus, using Lemma \ref{X6}, Lemma \ref{X3} and Lemma \ref{X11}, we obtain
  \begin{align}\label{X24}
   N_3=\begin{cases}
        80p^4(p^2-4p-1), &p\equiv 1\bmod 4;\\
        80p^4(p^2-4p+3), &p\equiv 3\bmod 4.
       \end{cases}
  \end{align}
We now consider two cases to evaluate $N_4$.\\
Case 1: If $p\equiv 1\bmod 4$, then we have
\begin{align}\label{X9}
 &\left(\sum_{a=2}^{p-2}\chi(a)\left(\frac{a^2-1}{p}\right)\right)^3\\
 &=\sum_{a=2}^{p-2}\chi(a)\left(\frac{a^2-1}{p}\right)\sum_{b=2}^{p-2}
 \chi(\overline{b})\left(\frac{b^2-1}{p}\right)\sum_{c=2}^{p-2}\chi(c)\left(\frac{c^2-1}{p}\right)\notag\\
 &=\left[(p-3)A+\sum_{a=2}^{p-2}\sum_{c=1}^{p-1}\chi(a)\left(\frac{a^2-c^2}{p}\right)\left(\frac{c^2-1}{p}\right)\right]\sum_{b=2}^{p-2}
 \chi(\overline{b})\left(\frac{b^2-1}{p}\right)\notag\\
 &=(p-3)A\sum_{b=2}^{p-2}\chi(b)\left(\frac{b^2-1}{p}\right)+\sum_{a=2}^{p-2}\sum_{b=2}^{p-2}\sum_{c=1}^{p-1}\chi(a\overline{b})\left(\frac{a^2-c^2}{p}\right)
 \left(\frac{c^2-1}{p}\right) \left(\frac{b^2-1}{p}\right)\notag
\end{align}
Case 2: If $p\equiv 3 \bmod 4$, then similarly we get
\begin{align}\label{X10}
 &\left(\sum_{a=2}^{p-2}\chi(a)\left(\frac{a^2-1}{p}\right)\right)^3\\
 &=-(p-3)A\sum_{b=2}^{p-2}\chi(b)\left(\frac{b^2-1}{p}\right)-\sum_{a=2}^{p-2}\sum_{b=2}^{p-2}\sum_{c=1}^{p-1}\chi(a\overline{b})\left(\frac{a^2-c^2}{p}\right)
 \left(\frac{c^2-1}{p}\right) \left(\frac{b^2-1}{p}\right).\notag
\end{align}
If $p\equiv 1\bmod 4$, then it follows from \eqref{X9}, Lemma \ref{X7}, \ref{X6}, \ref{X4} and \ref{X11} that,
\begin{align*}
&\sum_{\chi\neq\chi_0}\binom{5}{3}(Ap)^2\cdot\left(B\sum_{a=2}^{p-2}\chi(a)\left(\frac{a^2-1}{p}\right)\right)^3\\
 &=20p^3\sqrt{p}\left(\frac{n}{p}\right)\sum_{\chi\neq\chi_0}A\left(\sum_{a=2}^{p-2}\chi(a)\left(\frac{a^2-1}{p}\right)\right)^3\\
 &=20p^3\sqrt{p}\left(\frac{n}{p}\right)(p-3)\sum_{\chi\neq\chi_0}A^2\sum_{b=2}^{p-2}\chi(b)\left(\frac{b^2-1}{p}\right)\\
 &+20p^3\sqrt{p}\left(\frac{n}{p}\right)\sum_{a=2}^{p-2}\sum_{b=2}^{p-2}\sum_{c=1}^{p-1}\left(\frac{a^2-c^2}{p}\right)
 \left(\frac{c^2-1}{p}\right) \left(\frac{b^2-1}{p}\right)\left[\sum_{\chi\bmod p}\chi(a\overline{b})-1\right]\\
 &+20p^3\sqrt{p}\left(\frac{n}{p}\right)\sum_{a=2}^{p-2}\sum_{b=2}^{p-2}\sum_{c=1}^{p-1}\left(\frac{a^2-c^2}{p}\right)
 \left(\frac{c^2-1}{p}\right) \left(\frac{b^2-1}{p}\right)\left[\sum_{\chi\bmod p}\chi(-a\overline{b})-1\right]\\
 &=40p^3\sqrt{p}\left(\frac{n}{p}\right)[8+(p-1)N],
\end{align*}
where $N$ is given by \eqref{X1}. \\Now notice that for $p\equiv 3 \bmod 4 $, we have
\begin{align*}
&\sum_{a=1}^{p-1}\sum_{c=1}^{p-1}\left(\frac{a^2-c^2}{p}\right)
 \left(\frac{c^2-1}{p}\right) \left(\frac{a^2-1}{p}\right)\\
 &=\sum_{a=1}^{p-1}\sum_{c=1}^{p-1}\left(\frac{a^2c^2-c^2}{p}\right)
 \left(\frac{c^2-1}{p}\right) \left(\frac{a^2c^2-1}{p}\right)\\
 &=\sum_{a=1}^{p-1}\sum_{c=1}^{p-1}\left(\frac{a^2-1}{p}\right)
 \left(\frac{c^2-1}{p}\right) \left(\frac{a^2c^2-1}{p}\right)
\end{align*}
and
\begin{align*}
&\sum_{a=1}^{p-1}\sum_{c=1}^{p-1}\left(\frac{a^2-\overline{c}^2}{p}\right)
 \left(\frac{\overline{c}^2-1}{p}\right) \left(\frac{a^2-1}{p}\right)\\
 &=\sum_{a=1}^{p-1}\sum_{c=1}^{p-1}\left(\frac{a^2c^2-1}{p}\right)
 \left(\frac{1-c^2}{p}\right) \left(\frac{a^2-1}{p}\right)\\
 &=-\sum_{a=1}^{p-1}\sum_{c=1}^{p-1}\left(\frac{a^2-1}{p}\right)
 \left(\frac{c^2-1}{p}\right) \left(\frac{a^2c^2-1}{p}\right).
\end{align*}
Thus in this case
\begin{align}\label{X-13}
\sum_{a=1}^{p-1}\sum_{c=1}^{p-1}\left(\frac{a^2-c^2}{p}\right)
 \left(\frac{c^2-1}{p}\right) \left(\frac{a^2-1}{p}\right)=0.
\end{align}
Hence for $p\equiv 3\bmod 4$, using \eqref{X10}, \eqref{X-13} and Lemma \ref{X7}, \ref{X6} and \ref{X4} we obtain
\begin{align*}
&\sum_{\chi\neq\chi_0}\binom{5}{3}(Ap)^2\cdot\left(B\sum_{a=2}^{p-2}\chi(a)\left(\frac{a^2-1}{p}\right)\right)^3\\
 &=20p^2B^3\sum_{\chi\neq\chi_0}A\left(\sum_{a=2}^{p-2}\chi(a)\left(\frac{a^2-1}{p}\right)\right)^3\\
 &=-20p^2B^3(p-3)\sum_{\chi\neq\chi_0}A^2\sum_{b=2}^{p-2}\chi(b)\left(\frac{b^2-1}{p}\right)\\
 &-20p^2B^3\sum_{a=2}^{p-2}\sum_{b=2}^{p-2}\sum_{c=1}^{p-1}\left(\frac{a^2-c^2}{p}\right)
 \left(\frac{c^2-1}{p}\right) \left(\frac{b^2-1}{p}\right)\left[\sum_{\chi\bmod p}\chi(a\overline{b})-1\right]\\
 &-20p^2B^3\sum_{a=2}^{p-2}\sum_{b=2}^{p-2}\sum_{c=1}^{p-1}\left(\frac{a^2-c^2}{p}\right)
 \left(\frac{c^2-1}{p}\right) \left(\frac{b^2-1}{p}\right)\left[\sum_{\chi\bmod p}\chi(-a\overline{b})-1\right]\\
 &=0.
\end{align*}
Thus, we get
\begin{align*}
 N_4=\begin{cases}
     40p^3\sqrt{p}\left(\frac{n}{p}\right)[8+(p-1)N], &p\equiv 1\bmod 4;\\
     0, &p\equiv 3\bmod 4.
    \end{cases}
\end{align*}
Using Lemma \ref{sum-2}, we obtain
\begin{align}\label{X25}
 N_4=\begin{cases}
     O(p^{11/2}), &p\equiv 1\bmod 4;\\
     0, &p\equiv 3\bmod 4.
    \end{cases}
\end{align}
We next evaluate $N_5$. Using Lemma \ref {X11} we can write
   \begin{align}\label{X30}
   &\sum_{\chi\neq\chi_0}\binom{5}{4}(Ap).\left(B\sum_{a=2}^{p-2}\chi(a)\left(\frac{a^2-1}{p}\right)\right)^4\\\notag
   &=5p^3\left[\sum_{\chi\neq\chi_0}\left(\sum_{a=2}^{p-2}\chi(a)\left(\frac{a^2-1}{p}\right)\right)^4+\sum_{\chi\neq\chi_0}\chi(-1)\left(\sum_{a=2}^{p-2}\chi(a)\left(\frac{a^2-1}{p}\right)\right)^4\right].\notag
   \end{align}
   We note that for any non-principal character $\chi$ modulo an odd prime $p$, $A=0$ or $2$. Hence from Lemma $\ref{X2}$
one can see that if $p\equiv 1\bmod 4,$ then $\displaystyle\sum_{a=2}^{p-1}\chi(a)\left(\frac{a^2-1}{p}\right)$ is a real number;
while if $p\equiv 3 \bmod 4$, then $\displaystyle\sum_{a=2}^{p-1}\chi(a)\left(\frac{a^2-1}{p}\right)$ is a purely imaginary number.
   Using \eqref{X5}, we have
   \begin{align}\label{X31}
   &\sum_{\chi\neq\chi_0}\left(\sum_{a=2}^{p-2}\chi(a)\left(\frac{a^2-1}{p}\right)\right)^4\\\notag
   &=(p-3)^2\sum_{\chi\neq\chi_0}A^2\\\notag
   &+2(p-3)\left(\frac{-1}{p}\right)\left[\sum_{a=2}^{p-2}\sum_{b=1}^{p-1}\left(\frac{a^2-b^2}{p}\right)\left(\frac{b^2-1}{p}\right)\left(\sum_{\chi\bmod p}\chi(a)-1\right)\right.\\\notag
   &\left.+\sum_{a=2}^{p-2}\sum_{b=1}^{p-1}\left(\frac{a^2-b^2}{p}\right)\left(\frac{b^2-1}{p}\right)\left(\sum_{\chi\bmod p}\chi(-a)-1\right)\right]\\\notag
   &+\sum_{a=2}^{p-2}\sum_{b=1}^{p-1}\sum_{c=2}^{p-2}\sum_{d=1}^{p-1}\left(\frac{a^2-b^2}{p}\right)\left(\frac{b^2-1}{p}\right)\left(\frac{c^2-d^2}{p}\right)\left(\frac{d^2-1}{p}\right)\left[\sum_{\chi \bmod p}\chi(a\overline{c})-1\right]\notag
   \end{align}
   and
   \begin{align}\label{X32}
   &\sum_{\chi\neq\chi_0}\chi(-1)\left(\sum_{a=2}^{p-2}\chi(a)\left(\frac{a^2-1}{p}\right)\right)^4\\\notag
   &=2(p-3)^2\sum_{\chi\neq\chi_0}A\\\notag
   &+2(p-3)\left(\frac{-1}{p}\right)\left[\sum_{a=2}^{p-2}\sum_{b=1}^{p-1}\left(\frac{a^2-b^2}{p}\right)\left(\frac{b^2-1}{p}\right)\left(\sum_{\chi\bmod p}\chi(a)-1\right)\right.\\\notag
   &\left.+\sum_{a=2}^{p-2}\sum_{b=1}^{p-1}\left(\frac{a^2-b^2}{p}\right)\left(\frac{b^2-1}{p}\right)\left(\sum_{\chi\bmod p}\chi(-a)-1\right)\right]\\
   &+\sum_{a=2}^{p-2}\sum_{b=1}^{p-1}\sum_{c=2}^{p-2}\sum_{d=1}^{p-1}\left(\frac{a^2-b^2}{p}\right)\left(\frac{b^2-1}{p}\right)\left(\frac{c^2-d^2}{p}\right)\left(\frac{d^2-1}{p}\right)\left[\sum_{\chi \bmod p}\chi(-a\overline{c})-1\right].\notag
   \end{align}
   Hence from \eqref{X30}, \eqref{X31} and \eqref{X32} and using Lemma \ref{X6} and \ref{X3}, we obtain
  \begin{align*}
   N_5=\begin{cases}
   10p^3[2p^3-14p^2+30p-34+(p-1)T], &p\equiv 1\bmod 4;\\
    10p^3[2p^3-14p^2+30p-18+(p-1)T], &p\equiv 3\bmod 4,
 \end{cases}
  \end{align*}
  where
  \begin{align*}
  T=\sum_{a=2}^{p-2}\sum_{b=1}^{p-1}\sum_{d=1}^{p-1}\left(\frac{a^2-b^2}{p}\right)\left(\frac{b^2-1}{p}\right)\left(\frac{a^2-d^2}{p}\right)\left(\frac{d^2-1}{p}\right).
  \end{align*}
 Now from relation \cite[(6)]{BB}, we know $T=p^2+O(p^{3/2})$. Hence we get
  \begin{align}\label{X26}
   N_5=
   30p^6+O(p^{11/2}),
  \end{align}
  for all odd prime $p$. Finally using \eqref{X5} we have
\begin{align}\label{X35}
 &\sum_{\chi\neq\chi_0}\binom{5}{5}\left(B\sum_{a=2}^{p-2}\chi(a)\left(\frac{a^2-1}{p}\right)\right)^5\\\notag
 &=B^5\sum_{\chi\neq\chi_0}\left(\sum_{a=2}^{p-2}\chi(a)\left(\frac{a^2-1}{p}\right)\right)^4\left(\sum_{e=2}^{p-2}\chi(e)\left(\frac{e^2-1}{p}\right)\right)\\\notag
 &=B^5\sum_{\chi\neq\chi_0}\left(\left(\frac{-1}{p}\right)(p-3)A+\sum_{a=2}^{p-2}\sum_{b=1}^{p-1}\chi(a)\left(\frac{a^2-b^2}{p}\right)
  \left(\frac{b^2-1}{p}\right)\right)^2\left(\sum_{e=2}^{p-2}\chi(e)\left(\frac{e^2-1}{p}\right)\right)\\\notag
  &=B^5\left[\sum_{\chi\neq\chi_0}(p-3)^2A^2\sum_{e=2}^{p-2}\chi(e)\left(\frac{e^2-1}{p}\right)\right.\\\notag
  &+2(p-3)\sum_{a=2}^{p-2}\sum_{b=1}^{p-1}\sum_{e=2}^{p-2}\left(\frac{a^2-b^2}{p}\right)
  \left(\frac{b^2-1}{p}\right)\left(\frac{e^2-1}{p}\right)\left(\sum_{\chi\bmod p}\chi(a\overline{e})-1\right)\\\notag
  &+2(p-3)\sum_{a=2}^{p-2}\sum_{b=1}^{p-1}\sum_{e=2}^{p-2}\left(\frac{a^2-b^2}{p}\right)
  \left(\frac{b^2-1}{p}\right)\left(\frac{e^2-1}{p}\right)\left(\sum_{\chi\bmod p}\chi(-a\overline{e})-1\right)\\\notag
  &+\left(\frac{-1}{p}\right)\left.\sum_{a=2}^{p-2}\sum_{b=1}^{p-1}\sum_{c=2}^{p-2}\sum_{d=1}^{p-1}\sum_{e=2}^{p-2}\left(\frac{a^2-b^2}{p}\right)\left(\frac{b^2-1}{p}
  \right)\left(\frac{c^2-d^2}{p}\right)\left(\frac{d^2-1}{p}\right)\left(\frac{e^2-1}{p}
  \right)\right.\notag\\
  &\hspace{2cm}\left.\times \left(\sum_{\chi\bmod p}\chi(ac\overline{e})-1\right)\right].\notag
  \end{align}
 Again, from \eqref{X35}, using Remark \ref{remark}, Lemma \ref{X7}, \ref{X6}, \ref{X4} and \ref{X11}, we get
  \begin{align*}
  N_6=\begin{cases}
  \left(\frac{n}{p}\right)p^2\sqrt{p}\left[32+(p-1)S+4(p-3)(p-1)N\right],&p\equiv 1\bmod 4;\\
  0,&p\equiv 3\bmod 4,
  \end{cases}
  \end{align*}
  where
  \begin{align*}
  S=\sum_{a=2}^{p-2}\sum_{b=1}^{p-1}\sum_{c=2}^{p-2}\sum_{d=1}^{p-1}\left(\frac{a^2-b^2}{p}\right)\left(\frac{b^2-1}{p}\right)
  \left(\frac{c^2-d^2}{p}\right)\left(\frac{d^2-1}{p}\right)\left(\frac{a^2c^2-1}{p}\right).
  \end{align*}
 We have proved in Theorem 1.1 that $S=O(p^2)$. Hence using Lemma \ref{sum-2} and the estimate for $S$, we have the estimate
    \begin{align}\label{X27}
  N_6=\begin{cases}
  O(p^{11/2}),&p\equiv 1\bmod 4;\\
  0,&p\equiv 3\bmod
   4.
  \end{cases}
  \end{align}
 Combining $\eqref{X20}$, $\eqref{X21}$, $\eqref{X50}$, $\eqref{X22}$, $\eqref{X23}$, $\eqref{X24}$, $\eqref{X25}$, $\eqref{X26}$ and $\eqref{X27}$
 we get the desired result, Theorem 1.4.
 \end{proof}

\begin{proof}[Proof of Theorem  1.6]
  Applying the asymptotic formulas in [2], [5] and [10], our Lemma 2.3 and Lemma 2.9 with $m=1$  we have
\begin{eqnarray*}
&& \mathop{\sum_{\chi \bmod p}}_{\chi\neq \chi_0}\left|\sum_{a=1}^{p-1}\chi\left(ma+\overline{a}\right)\right|^4= \mathop{\sum_{\chi \bmod p}}_{\chi\neq \chi_0}\left|\sum_{a=1}^{p-1}\chi(a)\left(\frac{a^2-1}{p}\right)\right|^4\nonumber\\
&=&\frac{1}{\left|\tau(\chi_2)\right|^4}\cdot \mathop{\sum_{\chi(-1)=1}}_{\chi\neq \chi_0}\left(\left|\sum_{a=1}^{p-1}\chi(a)e\left(\frac{a^2}{p}\right)\right|^2-2p    \right)^4\nonumber\\
&=&\frac{1}{p^2}\cdot \sum_{k=0}^4\binom{4}{k}\mathop{\sum_{\chi(-1)=1}}_{\chi\neq \chi_0}\left|\sum_{a=1}^{p-1}\chi(a)e\left(\frac{a^2}{p}\right)\right|^{2k}(-2p)^{4-k}\nonumber\\
&=&3\cdot p^3+ O\left(p^{\frac{5}{2}}\cdot \ln p\right),
\end{eqnarray*}
where $\tau(\chi_2)=\displaystyle \sum_{a=1}^{p-1}\left(\frac{a}{p}\right)e\left(\frac{a}{p}\right)$ denotes the classical Gauss sums.

This proves Theorem 1.6.
 \end{proof}

It is clear that Theorem 1.5 and Theorem 1.7  follows from the estimate
$$
\sum_{a=1}^{p-1}\left|\sum_{\chi\neq \chi_0}\chi(a)\cdot \left|L(1,\chi)\right|\right|\ll p\cdot \ln p,
$$
the reference [10], Lemma 2.3 and the conventional methods.

 This completes the proofs of our all results.
 
\section{Acknowledgement}
During the preparation of this work, N.B. was supported by the postdoctoral fellowship, Harish-Chandra Research Institute, Prayagraj; and A.R. was partially supported by grants MTM2016-75027-P (Ministerio de Economı\'{\i}a y
Competitividad and FEDER) and US-1262169 (Consejería de Econom\'{\i}a, Conocimiento, Empresas y Universidad de la Junta de Andalucía and FEDER).



\begin{thebibliography}{99}

\bibitem{gauss}
T. M. Apostol, {\it Introduction to Analytic Number Theory}, Springer-Verlag, New York, 1976.

\bibitem{BB}
N. Bag, R. Barman {\it Higher Order Moments of Generalized Quadratic Gauss Sums Weighted by $L$-functions}, Asian Journal of Mathematics, accepted for publication.


\bibitem{cochrane}
T. Cochrane and Z. Y. Zheng, {\it Pure and mixed exponential sums}, Acta Arith. 91 (1999), 249-278.

\bibitem{FKMS}
\'{E}. Fouvry, E. Kowalsky, P. Michel, and W. Sawin, {\it Lectures on applied $\ell$-adic cohomology. (English Summary) Analytic methods in arithmetic geometry}, Contemp. Math., Amer. Math. Soc., Providence, RI (2009).

\bibitem{yuan}
Y. He and Q. Liao, {\it On an identity associated with Weil's estimate and its applications}, Journal of Number Theory 129 (2009), 1075-1089.

\bibitem{katz-esde}
N.M. Katz, {\it Exponential Sums and Differential Equations}, Annals of Mathematics Studies 124, Princeton University Press, 1990.

\bibitem{katz-sarnak}
N.M. Katz and P. Sarnak, {\it Random matrices, Frobenius eigenvalues, and monodromy}, Colloquium Publications 45, Amer. Math. Soc., Providence, RI (1999).

\bibitem{RL} A. Rojas-León, {\it Local convolution of $\ell$-adic sheaves on the torus}, Mathematische Zeitschrift, Vol. 274, No. 3 (2013), 1211--1230

\bibitem{weil}
A. Weil, {\it On some exponential sums}, Proc. Nat. Acad. Sci. U.S.A. 34 (1948), 203-210.

\bibitem{zhang}
 W.P. Zhang, {\it Moments of generalized quadratic Gauss sums weighted by $L$-functions},
Journal of Number Theory 92 (2002), 304-314.

\end{thebibliography}
\end{document}